\documentclass[12pt,letterpaper]{amsart}
\usepackage[breaklinks=true]{hyperref}
\usepackage{geometry}
\theoremstyle{plain}
\newtheorem{theorem}{Theorem}[section]
\newtheorem{lemma}{Lemma}[section]

\theoremstyle{definition}

\newtheorem{remark}{Remark}[section]

\newcommand{\comment}[1]{}

\DeclareMathOperator{\gr}{\ensuremath{gr}}

\DeclareMathOperator{\Sym}{\ensuremath{Sym}}

\DeclareMathOperator{\ind}{\ensuremath{index}}
\DeclareMathOperator{\Frac}{\ensuremath{Frac}}

\begin{document}
\title[KW1 for large characteristic]{A simple proof of the first Kac-Weisfeiler conjecture for algebraic Lie algebras in large characteristics}
\author{Akaki Tikaradze}
\email{ tikar06@gmail.com}
\address{University of Toledo, Department of Mathematics \& Statistics, 
Toledo, OH 43606, USA}

\maketitle
\begin{abstract}
Given a Lie algebra $\mathfrak{g}$ of an algebraic group over a ring $S,$
we show that the first Kac-Weisfeiler conjecture holds for reductions of $\mathfrak{g} \mod p$ for
large enough primes $p,$ reproving
a recent result of Martin, Stewart and Topley \cite{MST}. As a byproduct of our proof, we show that the center of
the  skew field of fractions of the the enveloping algebra $\mathfrak{U}\mathfrak{g}_{\bold{k}}$ for a field $\bold{k}$ of characteristic $p>>0$
is generated by the $p$-center and by the reduction $\mod p$ of the center of the fraction skew field of
$\mathfrak{U}\mathfrak{g}.$

\end{abstract}
\vspace{0.5in}

Recall that given a restricted Lie algebra  $\mathfrak{g}$ over an algebraically closed field
$\bold{k}$ of characteristic $p,$ the first Kac-Weisfeiler
conjecture  \cite{KW} (the KW1 conjecture for short) asserts that the largest dimension of an irreducible
$\mathfrak{g}$-module (to be denoted  from now on by $M(\mathfrak{g})$)  equals $p^{\frac{1}{2}(\dim \mathfrak{g}-\ind(\mathfrak{g}))}$, where
the index of a Lie algebra $\mathfrak{g}$ is defined to be as the smallest possible dimension of the stabilizer $\mathfrak{g}^{\chi}$
for $\chi\in \mathfrak{g}^*$ under the coadjoint action of $\mathfrak{g}$ on  $\mathfrak{g}^*.$
It is well-known that $M(\mathfrak{g})$ equals to the PI-degree of $\mathfrak{U}\mathfrak{g}$, which is also
equal the the rank of $\mathfrak{U}\mathfrak{g}$ over its center $Z(\mathfrak{U}\mathfrak{g}).$
It is also known [\cite{PS}, Remark 5.1] that the following inequality holds
$$p^{\frac{1}{2}(\dim \mathfrak{g}-\ind(\mathfrak{g}))}\leq M(\mathfrak{g}).$$
 It was  also established in \cite{PS} that
the KW1 conjecture  holds for Lie algebras $\mathfrak{g}$ possessing $\chi\in\mathfrak{g}^*$
such that $\mathfrak{g}^{\chi}$ is a toral subalgebra. 

In a very recent paper [\cite{MST}, Theorem 1.1],  the KW1 conjecture was established
for restricted Lie algebras in very large characteristics.
Namely, the authors proved that for any $n$ there exists $p(n)$ such that for any $p>p(n)$
and a restricted subalgebra $\mathfrak{g}\subset \mathfrak{gl}_n(\bold{k})$ over an algebraically closed field $\bold{k}$ of characteristic $p$, the KW1 conjecture holds
for $\mathfrak{g}$. From this it is deduced that given an algebraic Lie algebra $\mathfrak{g}$,
then the  KW1 conjecture holds for all $\mathfrak{g}_{\bold{k}}$ for all $\bold{k}$ of sufficiently large characteristic [\cite{MST}, Theorem 1.2].
 The goal of our short note is to give
an alternative, very short proof of the latter result.
 
 Throughout given a commutative (noncommutative Noetherian) domain $R,$ by $\Frac(R)$ we will denote
its quotient field (skew field) of fractions. 
As usual given a module $M$ over a commutative domain $R$, by the rank of $M$ over $R$ we
will understand the dimension of $M\otimes_R\Frac(R)$ over $\Frac(R).$
Also, given a $p$-restricted Lie algebra $\mathfrak{g}$ over a field $\bold{k}$, by $Z_p(\mathfrak{g})$
we will denote  the central $\bold{k}$-subalgebra of $\mathfrak{U}\mathfrak{g}$
 generated by $g^p-g^{[p]}, g\in \mathfrak{g}_{\bold{k}}.$
Also, for a Lie algebra $\mathfrak{g}$, by $D(\mathfrak{g})$ we will denote the skew filed of fractions of $\mathfrak{U}\mathfrak{g}.$
\begin{theorem}\label{main}
Let $S\subset\mathbb{C}$ be a finitely generated ring.
Let $\mathfrak{g}$ be a Lie algebra of an algebraic group over $S.$
Then for all $p>>0$ and base change $S\to \bold{k}$ to an algebraically closed field of characteristic
$p,$ the KW1 conjecture holds for $\mathfrak{g}_{\bold{k}}=\mathfrak{g}\otimes_S\bold{k}.$

\end{theorem}

Along the way of proving Theorem \ref{main}, we also show the following.
\begin{theorem}\label{center}
Let $ \mathfrak{g}$ be a Lie algebra of an algebraic group over $S\subset\mathbb{C.}$
Then  the center of $D(\mathfrak{g}_{\bold{k}})$ is  generated by
$Z_p(\mathfrak{g}_{\bold{k}})$ and the image of the center of $D(\mathfrak{g})$
in $Z(D(\mathfrak{g}_{\bold{k}})).$

\end{theorem}

For the proof we will need to recall couple of simple lemmas from commutative algebra.
\begin{lemma}\label{rank}
Let $S\subset\mathbb{C}$ be a finitely generated ring.
Let $A$ be a finitely generated commutative algebra over $S$ such that $A_{\mathbb{C}}$ is a domain.
 Let $B\subset A$ be a finitely generated  $S$-subalgebra.
Then for all $p>>0$ and a base change $S\to\bold{k}$ to an algebraically closed field $\bold{k}$ of characteristic $p$
the rank of $A_{\bold{k}}$ over $B_{\bold{k}}A_{\bold{k}}^p$ is  $p^{\dim(A)-\dim(B)}.$
\end{lemma}
\begin{proof}
By localizing $B$ if necessary, we may assume by the Noether  normalization lemma  that
$A$ is a finite module over $B[x_1,\cdots, x_n]$, where $ x_1,\cdots x_n\in A$ are algebraically independent over $B.$
Let $m$ be a number of generators of $A$ as a module over $B[x_1,\cdots, x_n].$ Clearly $n=\dim(A)-\dim(B).$
Now let $S\to\bold{k}$ be a base change to an algebraically closed field of characteristic $p>>0$ (in particular $p>m$.)
So $A_{\bold{k}}$ is a finite  module with (at most) $m$ generators over $B_{\bold{k}}[x_1,\cdots, x_n].$
Clearly the rank of $A_{\bold{k}}$ over $B_{\bold{k}}[x_1^p,\cdots, x_n^p]\subset B_{\bold{k}}A_{\bold{k}}^p$
is $p^nl,$ where $l\leq m$ is the rank of $A_{\bold{k}}$ over $B_{\bold{k}}[x_1,\cdots, x_n].$
On the other hand, the rank $A_{\bold{k}}$ over $A_{\bold{k}}^p$ is a power of $p.$
So, the rank of $A_{\bold{k}}$ over $B_{\bold{k}}A_{\bold{k}}^p$ is a power of $p$ and divides $p^nl$,
hence it must divide $p^n.$ However, the rank of $B_{\bold{k}}A_{\bold{k}}^p$ over $B_{\bold{k}}[x_1^p,\cdots, x_n^p]$
is at most $m.$ Thus the rank of $A_{\bold{k}}$ over $B_{\bold{k}}A_{\bold{k}}^p$ is  $p^{\dim(A)-\dim(B)}.$
\end{proof}

\begin{lemma}\label{Gr}
Let $R$ be a nonnegatively filtered commutative  algebra over a field $\bold{k}$ such that 
$\gr(R)$ is a finitely generated $\bold{k}$-domain.
Let $M$ be an $R$-module equipped with a compatible filtration such that $\gr(M)$ is a finitely generated
$\gr(R)$-module. Then the rank of $M$ over $R$ equals to the rank of $\gr(M)$ over $\gr(R).$
\end{lemma}
For a proof of Lemma \ref{Gr} see for example [\cite{T}, Lemma 2.3]. For a more general
result covering noncommutative algebras see [\cite{G}, Lemma 6.2].

\begin{proof}[Proof of Theorems \ref{main} and \ref{center}]
Let $G$ be a connected algebraic group corresponding to $\mathfrak{g}_{\mathbb{C}}.$
Let $m$ be the index of $\mathfrak{g}_{\mathbb{C}}.$
 As it is well-known $\mathbb{C}(\mathfrak{g}_{\mathbb{C}}^*)^G$ has the transcendence degree $m$ over $\mathbb{C}.$
Let 
$$\lbrace \frac{f_i}{g_i}\in \mathbb{C}(\mathfrak{g}_{\mathbb{C}}^*)^G, f_i, g_i\in \mathbb{C}[\mathfrak{g}^*_{\mathbb{C}}], 1\leq i\leq m\rbrace $$
 be algebraically independent elements,
written as reduced fractions. 
By localizing $S$ if necessary, we may assume that $f_i, g_i\in \Sym[\mathfrak{g}], 1\leq i\leq m.$
Denote by $\phi_i$ (respectively $\psi_i$) the image of $f_i$ (resp. $g_i$)
under the symmetrization map $\Sym(\mathfrak{g})\to \mathfrak{U}\mathfrak{g}$ (localizing $S$ further if necessary.)
Then we claim that $\frac{\phi_i}{\psi_i}\in Z(D(\mathfrak{g})), 1\leq i\leq m.$
Indeed, it follows that there is a character $\chi:G\to \mathbb{C}^*$, such that $f_i, g_i$
are semi-invariants of $G$ with respect to $\chi.$ Hence so are $\phi_i, \psi_i$. Which
implies that $\frac{\phi_i}{g_i}\in D(\mathfrak{g})^G.$
Let $p>>0$ and $S\to\bold{k}$ be a base change to an algebraically closed field $\bold{k}$ of characteristic $p.$
Then it follows from elementary linear algebra that $m=\ind(\mathfrak{g}_{\bold{k}}).$

To establish the KW1 conjecture for $\mathfrak{g}_{\bold{k}}$, we need to show that the rank of  $\mathfrak{U}(\mathfrak{g}_{\bold{k}})$ over
$Z(\mathfrak{U}\mathfrak{g}_{\bold{k}})$ is at most $p^{\dim\mathfrak{g}_{\bold{k}}-m}.$
For this purpose, it is enough to show that the  rank of $\Sym(\mathfrak{g}_{\bold{k}})$ as $\gr(Z(\mathfrak{U}\mathfrak{g}_{\bold{k}}))$
module is at most $p^{\dim\mathfrak{g}_{\bold{k}}-m}$ by Lemma  \ref{Gr}.
Let $z_i\in \mathfrak{U}\mathfrak{g}_{\bold{k}}$ be such that $\psi_iz_i\in Z(\mathfrak{U}\mathfrak{g}_{\bold{k}}).$ So $\phi_iz_i\in Z(\mathfrak{U}\mathfrak{g}_{\bold{k}}).$
Now since $\gr(\phi_i)\gr(z_i), \gr(\psi_i)\gr(z_i)\in \gr Z(\mathfrak{U}\mathfrak{g}_{\bold{k}}),$
it follows that 
$$\frac{f_i}{g_i}\in \Frac( \gr Z(\mathfrak{U}\mathfrak{g}_{\bold{k}})), 1\leq i\leq m.$$
 Put $A=\Sym(\mathfrak{g})[g_1^{-1},\cdots, g_m^{-1}]$ and $B=S[\frac{f_1}{g_1}, \cdots, \frac{f_m}{g_m}].$
As $\Frac(\gr Z(\mathfrak{U}\mathfrak{g}_{\bold{k}})$)  contains  $A_{\bold{k}}^pB_{\bold{k}},$
it suffices to show that the rank of $A_{\bold{k}}$ over $A_{\bold{k}}^pB_{\bold{k}}$ is at most $p^{\dim\mathfrak{g}_{\bold{k}}-m},$
which follows from  Lemma \ref{rank}.
To summarize, we proved that the degree of $D(\mathfrak{g}_{\bold{k}})$
over the field $K$ generated by $Z_p(\mathfrak{g})$ and $Z(D(\mathfrak{g}))$
is $p^{\dim \mathfrak{g}_{\bold{k}}-m}$. On the other hand, the degree of $D(\mathfrak{g}_{\bold{k}})$
over its center is at least $p^{\dim \mathfrak{g}_{\bold{k}}-m}.$ Thus $Z(D(\mathfrak{g}_{\bold{k}}))=K$
as desired.

\end{proof}
\begin{remark}
The statement of Theorem \ref{center}  and the following context for it was suggested  to us by L.Topley.
Let $\mathfrak{g}$ be a Lie algebra of an algebraic group over $\mathbb{Z}.$ Then it is natural
to ask whether the center of $\mathfrak{U}\mathfrak{g}_p$ is generated by
$Z_p(\mathfrak{g}_p)$ and  the reduction of the center of $\mathfrak{U}\mathfrak{g} \mod p$
(as conjectured by Kac in \cite{K}). Unfortunately this is not always true: Let $\mathfrak{g}$
be a three dimensional Lie algebra spanned by $h, x, y$ with relations 
$$[h, x]=nx, [h, y]=my, [x, y]=0, n, m\in\mathbb{Z}_{>0}, (n, m)=1.$$
Then the center of $\mathfrak{U}\mathfrak{g}_p$ is generated by $h^p-h, x^p, y^p, x^iy^j$
where $ ni+mj=0\mod p$ and $0<i, j<p.$ On the other hand the center of $ \mathfrak{U}\mathfrak{g}$
is trivial, hence $Z(\mathfrak{U}\mathfrak{g}_p)\neq Z_p(\mathfrak{g}_p).$
Meanwhile, the center of $D(\mathfrak{g})$ is generated by $x^my^{-n}$
which does generate $D(\mathfrak{g}_p)$ over $Z_p(\mathfrak{g}_p)$.
\end{remark}


\noindent\textbf{Acknowledgement:} I am very grateful to Lewis Topley for email correspondence.

\end{document}